\documentclass[11pt,a4paper]{amsart}
\usepackage{mathrsfs}
\usepackage{amsfonts}
\usepackage[notref,notcite]{showkeys}
\usepackage{enumerate}

\usepackage{amsmath,amssymb,xspace,amsthm}
\newtheorem{theorem}{Theorem}%[section]
\newtheorem{lemma}[theorem]{Lemma}%[section]
\newtheorem{definition}{Definition}%[section]
\numberwithin{equation}{section}

\def\Der{\operatorname{Der}}
\def\span{\operatorname{span}}

\newcommand{\p}{\partial}
\newcommand{\C}{\ensuremath{\mathbb C}\xspace}

\renewcommand{\a}{\ensuremath{\alpha}}

\newcommand{\h}{\ensuremath{\mathfrak{h}}}

\newcommand{\Z}{\ensuremath{\mathbb{Z}}\xspace}
\newcommand{\N}{\ensuremath{\mathbb{N}}\xspace}

\newcommand{\V}{\ensuremath{\mathfrak{V}}\xspace}
\newcommand{\W}{\ensuremath{\mathcal{W}}\xspace}

\renewcommand{\phi}{\varphi}

\def\sl{\mathfrak{sl}}

\def\A{\mathcal{A}}

\def\p{\partial}

\begin{document}
\title[$\W_n^+$- and $\W_n$-module structures  on $U(\h_n)$]{$\W_n^+$- and $\W_n$-module structures  on $U(\h_n)$ }
\author{Haijun Tan and Kaiming Zhao}
%\date{}
\maketitle

\begin{abstract} Let $\h_n$ be the Cartan subalgebra of  the  Witt algebras $\W_n^+=\text{Der}
\C[t_1, t_2, ..., t_n]$ and  $\W_n=\text{Der}\C[t_1^{\pm 1},t_2^{\pm
1},\cdots,t_n^{\pm1}]$
 where $1\le n\le \infty$. In this paper, we classify the modules
over  $\W_n^+$ and over $\W_n$ which are free $U(\h_n)$-modules of
rank $1$. These  are the $\W_n^+$-modules
$\Omega(\Lambda_{n},a, S) $ for some
$\Lambda_n=(\lambda_1,\cdots,\lambda_n) \in (\C^*)^n, a\in \C$, and
$S\subset \{1,2,..., n\}$;  and $\W_n$-modules  $\Omega(\Lambda_n,a)$
 for some $\Lambda_n\in (\C^*)^n$ and some $a\in \C.$
\end{abstract}


\vskip 5pt \noindent {\em 2000  Math. Subj. Class.:} 17B10, 17B20,
17B65, 17B66, 17B68

\vskip 10pt
\section{introduction}

\vskip 5pt We denote by $\mathbb{Z}$, $\mathbb{Z}_+$, $\mathbb{N}$
and $\mathbb{C}$ the sets of  all integers, non-negative integers,
positive integers and complex numbers, respectively.  All vector
spaces and algebras in this paper are over $\C$. We denote by
$U(\mathfrak{a})$ the universal enveloping algebra of the Lie
algebra $\mathfrak{a}$ over $\C$.

For any integer $n : 1\le n\le \infty$, let
$\A_n=\C[t_1^{\pm1},t_2^{\pm1},\cdots, t_n^{\pm1}]$ and
$\A_n^+=\C[t_1,t_2,\cdots, t_n]$ be the polynomial algebras. Then we
have the Lie algebras
 $\W_n=\Der(\A_n)$ and $\W_n^+=\Der(\A_n^+)$. These Lie  algebras are known as the Witt algebras of rank $n$.
 We know that  $\W_1$ is the centerless Virasoro algebra.

The theory of weight Virasoro modules with finite-dimensional weight
spaces is  fairly well-developed. In 1992, O. Mathieu \cite{M}
classified all irreducible modules with finite-dimensional weight
spaces over the Virasoro algebra, proving a conjecture of Kac
\cite{Ka}. More precisely, O. Mathieu proved that irreducible weight
modules with finite-dimensional weight spaces fall into two classes:
(i) highest/lowest weight modules and  (ii) modules of tensor fields
on a circle and their quotients. V. Mazorchuk and K. Zhao \cite{MZ1}
proved that an irreducible weight module over the Virasoro algebra
is either a Harish-Chandra module or a module in which all weight
spaces in the weight lattice are infinite-dimensional. In \cite{CGZ,
CM, LLZ, LZ2, Zh}, some simple weight modules  over the Virasoro
algebra with infinite-weight spaces were constructed. Very recently,
the non-weight simple representation theory of the Virasoro algebra
has made a big progress. A lot of new non-weight  simple modules
were obtained in \cite{BM, GLZ,LGZ, LLZ, LZ1, MW, MZ2, OW, TZ1, TZ2}
by using different methods.

The theory of weight representations over Witt algebras $\W_n$ for $
1<n<\infty$ has also well-developed. In 2013, Y. Billig and V.
Futorny \cite{BF} successfully proved the conjecture given by Eswara
Rao \cite{E} in $2004$. They proved that
 irreducible modules
over $\mathcal{W}_n$ $(1<n<\infty)$ with finite-dimensional weight
 spaces also fall in two classes: (i) modules of the highest weight type and  (ii) modules of tensor
 fields on a torus and their quotients.  Mazorchuk and Zhao \cite{MZ3} gave an explicit description of the support of an arbitrary irreducible weight module. There are also mixed weight modules over $\mathcal{W}_n$ for $n>1$, see, for example
 \cite{HWZ}. Recently, in \cite{TZ3} some examples of   non-weight irreducible  modules over $\mathcal{W}_n$ with $1<n<\infty$
 were given. We have not seen other examples of non-weight irreducible  representations over these Witt algebras.

For the Witt algebras $\W_n^+$, Penkov and Serganova \cite{PS} gave
an explicit description of the support of an arbitrary irreducible
weight module. We have not seen any other studies about irreducible
representations over these Witt algebras.

In this paper, we mainly concern about the non-weight
representations over Witt algebras $\W_n$ and $\W_n^+$ for $ 1\le
n\le \infty$. More precisely, we will classify a class of non-weight
modules which are free $U(\h_n)$-modules of rank $1$, where $\h_n$
is the Cartan subalgebra of $\W_n$  and $\W_n^+$.

\

Firstly, let us recall some definitions and   some old
representations over $\W_n$ and over $\W_n^{+}$. Then construct some
new representations over $\W_n$ and over $\W_n^{+}$.

 Denote by $\Z^{\infty}$ the set of all sequences $k=(k_1,k_2,k_3,\cdots)$ where $k_i\in \Z$ and only a finite number of  $k_i$ can be nonzero. Denote $(\C^*)^{\infty}=\{(c_k)_{k\in\N} : c_k\in\C^*\}$. Let $\Z_{\ge -1}=\{l\in \Z: l\ge -1\}$. For any $n\in \N$, let
  $$\Z^n_{+,i}=\{(k_1,\cdots,k_n)\in \Z^n: k_i\in \Z_{\ge -1} {\text{ and }} k_j\in \Z_+ {\text{ for all }}j\ne i\}, $$ $$\Z^{\infty}_{+,i}=\{(k_1,k_2,k_3,\cdots)\in \Z^{\infty}: k_i\in \Z_{\ge -1} {\text{ and }}  k_j\in \Z_+{\text{ for all }}j\ne i\}.$$

Let $\partial_i=t_i\frac{\partial}{\partial t_i}, i=1,2,\cdots, n.$
Then
$$\W_n=\bigoplus_{i=1}^n\A_n \partial_i.$$
 For $r=(r_1,\cdots,
r_n)\in \Z^n,$ set $t^r=t_1^{r_1}t_2^{r_2}\cdots t_n^{r_n}$. We can
write the Lie brackets in $\W_n$ as follows:
$$[t^r\partial_i,t^s\partial_j]=s_it^{r+s}\partial_j-r_jt^{r+s}\partial_i, \,\,\forall \,\, i,j=1,2,\cdots, n; r,s\in \Z^n.$$
It is known that  $\mathfrak{h}_{n}=\oplus_{i=1}^n\C \partial_i$ is
the
 {\emph {Cartan subalgebra}} of $\mathcal{W}_n$.

Similarly, the Cartan subalgebra of $\W_{\infty}$ is $\h_{\infty}=\oplus_{i=1}^{\infty}\C\p_i$.
For each $n\in \N$, $\W_n$ can be seen as a Lie subalgebra of
$\W_{\infty}$.
Clearly, $\W_n^+$(resp. $\W_{\infty}^{+}$) can be seen as a Lie
subalgebra of $\W_n$(resp. $\W_{\infty}$),
 and $\W_n^+$(resp. $\W_{\infty}^+$) and $\W_n$(resp. $\W_{\infty}$) share the common Cartan subalgebra $\h_{n}$(resp. $\h_{\infty}$).

With respect to the basis $\{\partial_i:i=1,2,...,n\}$, the weight
sets of $\W_n$ and $\W_n^+$ are $\Z^n$ and $\cup_{i=1}^n\Z_{+,i}^n$
respectively. For each $i\in \N$, denote $W_i^+=\text{Der}\C[t_i]$.

 The{\emph{ Virasoro algebra $\V$ } is the universal central extension of $\W_1$. For convenience,
 we denote $t_1^i\partial_1$ by $d_i$ for all $i\in \Z.$ Then
 $$\V=\big(\bigoplus_{i\in \Z}\C d_i\big)\oplus \C C,$$
 and the Lie bracket of $\V$ is defined by
 $$[C,d_i]=0,\ [d_i,d_j]=(j-i)d_{i+j}+\delta_{i,-j}\frac{i^3-i}{12}C,\ i,j\in \Z.$$

 Now we recall and define some representations over $\W_n$ and over $\W_n^+$.

For $\lambda\in \C^*,a\in \C,$ denote by $\Omega(\lambda,a)=\C[x]$
the polynomial associative algebra over $\C$ in indeterminate $x$.
In \cite{LZ1}, a class of $\V$-modules is defined  on
$\Omega(\lambda,a)$ by $$Cf(x)=0, d_if(x)=\lambda^i(x-i(a+1))f(x-i),
\forall i\in \Z, f(x)\in \C[x].$$ From \cite{LZ1} we know that
$\Omega(\lambda,a)$ is irreducible if and only if $a\ne-1.$ The
modules $\Omega(1, a)$ were also studied in \cite{BMZ}.

Since $\W_1^+$ is a Lie subalgebra of $\V$, each $\V$-module
$\Omega(\lambda, a)$ can be seen as a $\W_1^+$-module. We also
denote the corresponding $\W_1^+$-module by $\Omega(\lambda,a)$. It
is easy to see that $\W_1^+$-module $\Omega(\lambda,a)$ is
irreducible if and only if $a\ne -1.$

Now let us define a  class of new $\W_1^+$-modules.

\begin{definition} Let $\lambda\in\C^*,a\in \C$. We   have  a $\W_1^+$-module structure  on $\C[x]$ by
\begin{equation}\begin{split}&d_{-1}\cdot f(x)=\lambda^{-1}f(x+1),\ d_0\cdot f(x)=xf(x),\\
&d_k\cdot
f(x)=\lambda^kf(x-k)(x-k(a+1))\Pi_{i=0}^{k-1}(x+a-i),\end{split}\end{equation}
where $f(x)\in \C[x]$ and $ k\in \N.$ We denote this
$\W_1^+$-module by  $\Gamma(\lambda,a)$.
\end{definition}

 It is tedious but straightforward to verify that the above action makes $\C[x]$ into a $\W_1^+$-module.  Clearly, $\Gamma(\lambda,0)\cong \Gamma(\lambda,-1)$.
%$\Gamma(\lambda,a)$ is an irreducible $\W_1^+$-module, and

In \cite{TZ3}, a class of $\W_n$-modules  were defined as follows,
where $1<n<\infty$. For  any $a\in \C$ and
$\Lambda_n=(\lambda_1,\lambda_2,\cdots,\lambda_n)\in (\C^*)^n$,
denote by $\Omega(\Lambda_n,a)=\C[x_1,x_2,\cdots,$
 $x_n]$ the polynomial  algebra over $\C$ in commuting indeterminates $x_1,x_2,\cdots,
 x_n.$ The action of $\W_n$ on $\Omega(\Lambda_n,a)$ is defined by
\begin{equation}t^k\partial_i \cdot
f(x_1,\cdots,x_n)=\Lambda_n^k(x_i-k_i(a+1))f(x_1-k_1,\cdots,x_n-k_n),\end{equation}
where $k=(k_1,k_2,\cdots,k_n)\in \Z^n,f(x_1,\cdots,x_n)\in \A^n,
\Lambda_n^k=\lambda_1^{k_1}\lambda_2^{k_2}\cdots\lambda_n^{k_n},$
$i=1,2,\cdots,n.$ %From Example 3 in \cite{TZ3} we know that $\Omega(\Lambda_n, a)$ is irreducible if and only if $a\ne -1.$

Since $\W_n^+$ is a subalgebra of $\W_n$, $\Omega(\Lambda_n, a)$
can be seen as a $\W_n^+$-module. We also denote the corresponding
$\W_n^+$-module by $\Omega(\Lambda_n,a)$. %Clearly, as $\W_n^+$-module $\Omega(\Lambda_n,a)$ is irreducible if and only if $a\ne -1$.

For convenience, we make the following convention
$\Pi_{p=0}^{-1}g(p)=1$ for any function $g(p)$.
 Now we can define the following $\W_n^+$-modules for any $n\in \Z, $ $1<n<\infty$.

\begin{definition} Let $n\in \N,$ $\Lambda_n=(\lambda_1,\cdots,\lambda_n)\in (\C^*)^n, a\in \C$, and $S\subset \{1,2,..., n\}$. For any $k\in \cup_{j=1}^n\Z_{+, j}^n$, we define
$$\phi_{S,a}(k,i)=(x_i-k_i(a+1))\Pi_{q\in S}\Pi_{p=0}^{k_q-1}(x_q+a-p), \,\,{\rm if }\,\, k_i\ne-1\,\, {\rm or}\,\,
 i\notin S$$
$$\phi_{S,a}(k,i)=\Pi_{q\in S\setminus\{i\}}\Pi_{p=0}^{k_q-1}(x_q+a-p),\,\, {\rm if }\,\, k_i=-1\,\, {\rm and } \,\,i\in S.$$
Define the action of $\W_n^+$ on $\C[x_1,\cdots,x_n]$ as follows:
\begin{equation}
t^k\p_i\cdot
f(x_1,\cdots,x_n)=\Lambda_n^kf(x_1-k_1,\cdots,x_n-k_n)\phi_{S,a}(k,i),
\end{equation}
where $k=(k_1,k_2,\cdots,k_n)\in\Z_{+, i}^n, 1\le i\le n$ and
$f(x_1,\cdots,x_n)\in\A^+_n$. We denote this   $\W_n^+$-module by
$\Omega(\Lambda_n,a, S)$.
\end{definition}

 It is tedious but straightforward to verify that the above action makes $\A_n^+$ into a
 $\W_n^+$-module. We remark that when $S$ is empty,
$\Omega(\Lambda_n,a, S)\cong\Omega(\Lambda_n,a)$.
One can easily see that in Definition 2, $n$ can be replaced by
$\infty$.  Thus  we have the $\W_{\infty}^+$-module
$\Omega(\Lambda_\infty,a, S)$.

\

 The present paper is organized as follows.   In section 2, we classify
the modules over $\W_1^+$ which are free $U(\C d_0)$-modules of rank
$1$(Theorem 2). These modules are exactly $\Omega(\lambda,a)$ and
$\Gamma(\lambda,a)$ for some $\lambda\in \C^*, a\in \C$. As a
corollary, we also classify the modules of the Virasoro algebra
which  are free $U(\C d_0)$-modules of rank $1$ (Theorem 3). These
modules  are exactly $\Omega(\lambda,a)$ for some $\lambda\in \C^*$
and $a\in \C.$ We separate the Virasoro algebra case from $\W_n$
because of the center
$\C C$ of $\V$.
In section 3, we classify the modules over $\W_n^+$ which are free
$U(\h_n)$-modules of rank $1$ (Theorem 8). These  are exactly the
$\W_n^+$-modules $\Omega(\Lambda_{n},a, S) $ for some
$\Lambda_n=(\lambda_1,\cdots,\lambda_n)\in (\C^*)^n, a\in \C$, and
$S\subset \{1,2,..., n\}$.  We also classify the $\W_n$-modules
which are free $U(\h_n)$-modules of rank $1$ (Theorem 9). These
modules are $\Omega(\Lambda_n,a)$ for some $\Lambda_n\in (\C^*)^n$
and some $a\in \C.$ These results are also true for $n=\infty$.

\section{$\W_1^+$- and $\W_1$-module structures on   $\C[ d_0]$}

In this section, we will classify the modules over $\W_1^+$ and over
$\W_1$ which are  free
 $\C[ d_0]$-modules of rank $1$.

Let $1\le n\le \infty$ and let $M=U(\h_n)v$ be a $\W^+_n$-module
with left multiplication action of $\h_n$. The following observation
is important.
\begin{lemma} For any $f(\partial_1,\partial_2,...,\partial_n)\in
M$,    $1\le i\le n,$ $k\in \Z^n_{+,i}$, we have
\begin{equation}t^k\partial_i\cdot
f(\partial_1,\partial_2,...,\partial_n)v
=f(\partial_1-k_1,\partial_2-k_2,...,\partial_n-k_n)(t^k\partial_i\cdot
v),
\end{equation} and $t^k\p_i\cdot v\ne0$ for any $k\in \Z^n_{+,i}$.
\end{lemma}
\begin{proof} Formula (2.1) follows from $(t^k\partial_i)
\partial_j=(\partial_j-k_j)(t^k\partial_i)$ $\in U(\W_n^+)$ for any $1\le i,j\le
n$.

Now assume that $t^k\p_i\cdot v=0$ for some $k\ne (0,0,\cdots,0)$.
Then  $t^{k}\partial_i\cdot M=0,$ i.e., ann$(M)\ne0$. Since $\W_n^+$
is a simple Lie algebra and ann$(M)$ is a nonzero ideal of $\W_n^+$,
this means ann$(M)=\W_n^+$, hence $\h_n\cdot M=0$, which is
impossible. So $t^k\p_i\cdot v\ne 0$ for all $k\in \Z^{n}_{+,i}$. This completes the proof.
\end{proof}

Now we determine the $\W_1^+$-module structures on $M=\C[d_0]v$.

\begin{theorem}Let $M=\C[ d_0]$ be a $\W_1^+$-module with
left multiplication action of $d_0$.
 Then either $M\cong \Omega(\lambda,a)$ for some $\lambda\in \C^*,a\in \C$
 or $M\cong \Gamma(\lambda,a)$ for some $\lambda\in \C^*, a\in \C.$
\end{theorem}

\begin{proof} Since  $\W_1^+$ as a Lie algebra is generated by $d_{\pm1},d_2$, by
Lemma 1 we know that  the module structure of $M$ is determined by
$d_{\pm1}\cdot1$ and $d_{2}\cdot1$. Then there are polynomials $
f_i(d_0)\in\C[d_0]$ for $ i=\pm 1,2$ such  that
$d_{i}\cdot1=f_{i}(d_0)$.

Since $\C d_1+\C d_0 +\C d_{-1}\cong \sl_2(\C)$ as Lie algebra with
Cartan subalgebra $\C d_0$,   from \cite{N} we know that there are
three cases to consider for some $\lambda\in\C^*$ and $a\in\C$:

(i). $d_{1}(1)=\lambda(d_0- a)v$ and $d_{- 1}(v)=\lambda^{-1}(d_0+
a)$;

(ii). $d_{-1}\cdot1=\lambda^{- 1} (d_0- a)(d_0+ (a+1))$ and $d_{
1}\cdot1=\lambda $;

(iii). $d_{1}\cdot1=\lambda (d_0+ a)(d_0-(a+1))$ and
$d_{-1}\cdot1=\lambda^{-1}$.

If Case (i) holds, then by the equality $3d_1=[d_{-1},d_2]$ we
deduce that
$$3\lambda(d_0-a)=\lambda^{-1}((d_0+a)f_2(d_0+1)-(d_0+a-2)f_2(d_0))$$
$$=\lambda^{-1}((d_0+a)(f_2(d_0+1)-f_2(d_0))+2f_2(d_0)).$$
We see that the degree of $f_2$ has to be $1$. Let $f_2=\lambda^2(\a
d_0+\beta)$ for some $\a, \beta\in\C$. We know that $3(d_0-a)=3\a
d_0+(a\a+2\beta),$ yielding that $\a=1, \beta=-2a$, i.e., $d_{
2}v=\lambda^{2}(d_0-2a)v.$ Clearly, the actions of $d_{\pm 1},d_2$
on $M$
 coincide with the actions of $d_{\pm 1},d_2$ on $\Omega(\lambda,a-1)$, we see that $M\cong \Omega(\lambda,a-1)$.

Now suppose Case (ii) holds. From  $3d_1\cdot1= [d_{-1},d_2]\cdot1 $
we deduce that \begin{equation}3\lambda
=\lambda^{-1}((d_0-a)(d_0+a+1)f_2(d_0+1)\end{equation}
$$-(d_0-a-2)(d_0+a-1)f_2(d_0))$$
$$=\lambda^{-1}((d_0-a)(d_0+a+1)(f_2(d_0+1)-f_2(d_0))+2(2d_0-1)f_2(d_0))$$
Clearly, $f_2(d_0)$ cannot be a constant. Let the leading term of
$f_2(d_0)$ be $a_sd_0^s$ with $s>0$. The the leading term of the
right-hand side of (2.4) is $\lambda^{-1}a_s(s+4)d_0^{s+1}v\ne 0.$
So (2.4) cannot be true, i.e., Case (ii) does not occur.

Let us consider Case (iii). From $3d_1\cdot1= [d_{-1},d_2]\cdot1 $
we deduce that
$$3\lambda(d_0+a)(d_0-(a+1))
=\lambda^{-1}(f_2(d_0+1)-f_2(d_0)).$$So
\begin{equation}f_2(d_0)=\lambda^2((d_0+a)(d_0+a-1)(d_0-2(a+1))+b),\ b\in \C.\end{equation}

We first consider the case  $b=0$ in (2.5). In this case, denote
$$d_k\cdot1=\lambda^kf_{k,0}(d_0),\ k\in \N,$$ where $f_{k,0}(d_0)\in
\C[d_0]$. We see that the actions of $d_{\pm 1},d_2$ on $M$ coincide
with their actions on $\Gamma(\lambda,a)$. Then $M\cong
\Gamma(\lambda,a)$. In particular,
$$f_{k,0}(d_0)=\big(\Pi_{i=0}^{k-1}(d_0+a-i)\big)(d_0-k(a+1)), k\in \N.$$

Now we consider the case  $b\ne0$ in (2.5). Let
$$d_k\cdot1=\lambda^kf_{k}(d_0)=\lambda^k(f_{k,0}(d_0)+g_{k}(d_0)),k\in \N,$$
where $f_{k}(d_0), f_{k,0}(d_0),g_{k}(d_0)\in \C[d_0]$ with
$g_{1}(d_0)=0$ and $g_{2}(d_0)=b$. For $2\le k\le 5,$ by
$d_{k+1}v=\frac{1}{k-1}[d_1,d_k]v$, i.e.,
$$(k-1)g_{k+1}(d_0)=g_k(d_0-1)f_1(d_0)-g_k(d_0)f_1(d_0-k),$$ we deduce that
$$g_{3}(d_0)= 2b(2d_0-3),$$
$$g_{4}(d_0)= 2b(5d_0^2-20d_0+(18+a+a^2)),$$
$$g_{5}(d_0)= 2b\big(10d_0^3-75d_0^2+(173+6a+6a^2)d_0
-(120+15a+15a^2)\big).$$ And $d_5v=[d_2,d_3]v$ means that
$$\begin{aligned}g_{5}&(d_0)=f_{2,0}(d_0)g_3(d_0-2)+f_{3,0}(d_0-2)g_2(d_0)+g_{2}(d_0)g_3(d_0-2)\\
&-f_{3,0}(d_0)g_2(d_0-3)-f_{2,0}(d_0-3)g_3(d_0)-g_{3}(d_0)g_2(d_0-3)\\
=&f_{2,0}(d_0)2b(2(d_0-2)-3)+f_{3,0}(d_0-2)b+b2b(2(d_0-2)-3)\\
&-f_{3,0}(d_0)b-f_{2,0}(d_0-3)2b(2d_0-3)-2b(2d_0-3)b\\
=&(d_0+a)  (d_0+a-1)  (d_0-2 (a+1))2b (2 (d_0-2)-3)\\
&+ (d_0-2+a)  (d_0-2+a-1)  (d_0-2+a-2)   (d_0-2-3 (a+1))b\\
&+b2b (2 (d_0-2)-3) -(d_0+a)  (d_0+a-1) (d_0+a-2)   (d_0-3 (a+1)) b\\
&-(d_0-3+a)  (d_0-3+a-1)  (d_0-3-2 (a+1)) 2b (2d_0-3)\\ &-2b (2d_0-3) b\\
=&
2b\big(10d_0^3-75d_0^2+(173+6a+6a^2)d_0-(4b+120+7a-9a^2-16a^3)\big).
\end{aligned}$$
From the above  two expressions of $f_{5}(d_0)$ we deduce that
$$8b(b-(4a^3+6a^2+2a))=0,$$ yielding that  $b=4a^3+6a^2+2a$.
By simple computation, we obtain that
$$\begin{aligned}f_{2}(d_0)&=f_{2,0}(d_0)+(4a^3+6a^2+2a)\\
&=(d_0-(a+1))(d_0-(a+1)-1)(d_0+2a).\end{aligned}$$ Let $a'=-(a+1),$
then
$$f_{1}(d_0)=(d_0+a')(d_0-(a'+1)),$$$$f_{2}(d_0)=(d_0+a')(d_0+a'-1)(d_0-2(a'+1)).$$
Note that $d_{-1}\cdot1=\lambda^{-1}$, we know that in this case
$M\cong \Gamma(\lambda,a')$. This completes the proof.
\end{proof}

Now let us consider the $\V$-modules which are free $\C[d_0]$-modules of rank $1$. We have the following

\begin{theorem}Let $M$ be a $\V$-module which is a free $\C[ d_0]$-module  of rank $1$.
 Then $M\cong \Omega(\lambda,a)$ for some $a\in \C$ and $ \lambda\in \C^*$.
\end{theorem}
\begin{proof} Since $M$ is a free $\C[ d_0]$-module  of rank $1$, we may
assume that $M=\C[d_0]$ with left multiplication action of $d_0$. We
know that $\V$ has a subalgebra isomorphic to $W_1^+$:
$L=\span\{d_i:i\ge -1\}$. Using Theorem 2 we know that, there are
some $\lambda\in\C^*$ and $a\in\C$ such that
\begin{equation}d_if(d_0)=\lambda^i(d_0-i(a+1))f(d_0-i), \forall i\ge-1,
f(d_0)\in M,\end{equation} or (1.1) holds by replacing $x$ with
$d_0$.

Let $d_{-2}\cdot1=f_{-2}(d_0)$  for $f_{-2}(d_0)\in\C[d_0]$.

First assume that (1.1) holds. Using the same arguments as in the
paragraph containing Equation (2.2) by replacing $d_{\pm 1}, d_2$
with $d_{\mp 1},d_{-2}$ and by $-3d_{-1}\cdot1=[d_1,d_{-2}]\cdot1$
we can show that case (1.1) can not hold.

Thus only (2.4) holds.  From $-3d_{-1}v=[d_{1},d_{-2}]v$ we see that
$$-3\lambda^{-1}(d_0+a+1)v=\lambda((d_0-a-1)f_{-2}(d_0-1)-(d_0-a+1)f_{-2}(d_0))v$$
$$=\lambda((d_0-a-1)(f_{-2}(d_0-1)-f_{-2}(d_0))-2f_{-2}(d_0))v.$$
We see that the degree of $f_{-2}$ has to be $1$. Let
$f_{-2}=\lambda^{-2}(\a d_0+\beta)$ for some $\a, \beta\in\C$. We
know that $-3(d_0+a+1)=-3\a d_0+((a+1)\a-2\beta),$ yielding that $\a=1,
\beta=2(a+1)$, i.e., $d_{ -2}v=\lambda^{-2}(d_0+2(a+1))v.$

 From
$4d_0v-\frac{2^3-2}{12}Cv=[d_{-2},d_2]v$ we deduce that $Cv=0.$ Thus
we see that the actions
of $C, d_{\pm1},d_{\pm2}$ on $M$ coincide with their actions on $\Omega(\lambda,a)$, which means that $M\cong \Omega(\lambda,a)$.
This completes the proof.\end{proof}

\section{$\W_n^+$- and $\W_n$-module structures on $U(\h_n)$}

In this section, we determine the module structures on $U(\h_n)$
over $\W_n^+$ and over $\W_n$ for $1<n<\infty$. For convenience, we
denote
  the algebra
 $\C[\partial_1,\cdots,\partial_{i-1}, \partial_{i+1},\cdots,\partial_n]$ by $R_i$,   the algebra
$\C[\partial_1,\cdots,\partial_{i-1},\partial_{i+1},\cdots,$
$\partial_{j-1},\partial_{j+1},\cdots,
\partial_n]$ by $R_{ij}$ where $1\le i\ne j\le n.$

Let $M$ be a $\W_n^+$-module which is a free $U(\h_n)$-module of
rank $1$. We may assume  that $M=U(\h_n)v$ for some nonzero element
$v\in M$. We first prove the following auxiliary  result

\begin{lemma} We have  $t_i^{k_i}\p_i\cdot v
\in \C[\p_i]v$ for any $ k_i\in \Z_{\ge -1}$ and $ \ 1\le i\le n$.
\end{lemma}
\begin{proof}
From Lemma 1, we may assume that $$t_i^{k_i}\partial_i\cdot
v=h_{k_i}(\partial_i)v=\sum_{p=0}^{q_{k_i}}\nu_{k_i,p}\partial_i^{p}v,$$
for $k_i\in \Z_{\ge -1}$ where $ h_{k_i}(\partial_i)\in
R_i[\partial_i]$ with $\nu_{ k_i,p}\in R_i$ and $\nu_{ k_i,q_{
k_i}}\ne 0.$ Then
\begin{equation}\begin{split}2\p_iv=[t_i^{-1}\partial_i, t_i\partial_i]\cdot v
=\big(h_{1}(\partial_i+1)h_{-1}(\partial_i)-h_{-{1}}(\partial_i-1)h_{1}(\partial_i)\big)v.
\end{split}
\end{equation}Clearly,  $q_{1}+q_{-1}\le 1$ can not hold. So $q_1+q_{-1}\ge 2$. And the leading term  of $h_{1}(\partial_i+1)h_{-1}(\partial_i)-h_{-{1}}(\partial_i-1)h_{1}(\partial_i)$ is $
\nu_{1q_1}\nu_{-1,q_{-1}}(q_1+q_{-1})\p_i^{q_1+q_{-1}-1},$ which
forces that $q_1+q_{-1}=2$ and $\nu_{1q_1},\nu_{-1,q_{-1}}\in \C^*.$
Denote $\nu_{1,q_1}=\lambda_i,$ then
$\nu_{-1,q_{-1}}=\lambda_i^{-1}$.

There are three cases: (i). $q_1=q_{-1}=1$; (ii). $q_1=2,q_{-1}=0$;
(iii). $q_1=0,q_{-1}=2.$

\

{\bf{Case (i)}}. In this case, denote
$h_1(\p_i)=\lambda_i(\p_i-a_i)$, where $a_i\in R_i$, then (3.1)
implies that $h_{-1}(\p_i)=\lambda_i(\p_i+a_i)$. Since
$$3t_i\p_i\cdot v=[t_i^{-1}\p_i, t_i^{2}\p_i]\cdot
v=\lambda_i^{-1}((\p_i+a_i)(h_2(\p_i+1)-h_2(\p_i))+2h_2(\p_i))v,
$$that is,
$$3\lambda_i(\p_i-a_i)v=(\lambda_i^{-1}\nu_{2,q_2}(q_2+2)\p_i^{q_2}+g(\p_i))v,
$$where $g(\p_i)\in R_i[\p_i]$ has degree $\le q_2-1$, we deduce that $q_2=1$ and
$h_2(\p_i)=\lambda_i^2(\p_i-2a_i)$. Then, inductively, by
$t_i^{k_i+1}\p_i\cdot v=\frac{1}{k_i-1}[t_i\p_i, t_i^{k_i}\p_i]\cdot
v$, we obtain $t_i^{k_i}\p_i\cdot v=\lambda_i^{k_i}(\p_i-k_ia_i)v,\
k_i>2.$

\

{\bf{Case (ii)}}. In this case, from $t_i^{-1}\p_i\cdot
v=\lambda^{-1}_iv$ and by (3.1)  we have $t_i\p_i\cdot
v=\lambda_i(\p_i^2-\p_i+b_i)$ for some $b_i\in R_i.$ Since
$$3t_i\p_i\cdot v=[t_i^{-1}\p_i, t_i^{2}\p_i]\cdot v=\lambda_i^{-1}(h_2(\p_i+1)-h_2(\p_i))v,$$
by simple computations we see that
$$h_2(\p_i)=\lambda_i^2(\p_i^3-3\p_i^2+(2+3b_i)\p_i+b'_i)$$ for some $b'_i\in R_i.$ Again, inductively,
by $t_i^{k_i+1}\p_i\cdot v=\frac{1}{k_i-1}[t_i\p_i,
t_i^{k_i}\p_i]\cdot v$, we see that $t_i^{k_i}\p_i\cdot v\in
\C[\p_i,b_i,b_i']v$ for all $ k_i>2$.

\

Case (iii) will not occur. Otherwise, $h_1(\p_i)=\lambda_i$. From
(3.1) we deduce that $
h_{-1}(\p_i)=\lambda_i^{-1}(\p_i^2+\p_i+c_i)v$ for some $c_i\in
R_i$. So $3t_i\p_i\cdot v=[t_i^{-1}\p_i, t_i^2\p_i]\cdot v$ means
that
$$3\lambda_iv=(h_2(\p_i+1)h_{-1}(\p_i)-h_{-1}(\p_i-2)h_{2}(\p_i))v.$$
But in the above equality, the leading term of the right hand side
is $\lambda_i^{-1}\nu_{2,q_2}(q_2+4)\p_i^{q_2+1}$ with $q_2\ge 0$,
which is impossible.%and, clearly, the equality can not hold.

We have proved the following

{\bf{Claim 1.}} For any $1\le i\le n$ we have$$t_i^{k_i}\p_i\cdot
v=\lambda_i^{k_i}(\p_i-k_ia_i)v,\ k_i>-1,\,\,{\rm or}$$
$$t_i^{-1}\p_i\cdot v=\lambda^{-1}_iv,\,\,t_i\p_i\cdot
v=\lambda_i(\p_i^2-\p_i+b_i),\,\,{\rm and}$$
$$ t_i^{2}\p_i\cdot v=\lambda_i^2(\p_i^3-3\p_i^2+(2+3b_i)\p_i+b'_i)v$$for some $a_i, b_i, b'_i\in R_i.$

 So, to prove the lemma, we only need to show the following

{\bf{Claim 2.}} The coefficients $a_i,b_i, b_i'\in \C$ for all $i$.

Assume $a_i\notin \C,$ then there must be some $j$ with $j\ne i$
such that $a_i=F_i(\p_j)=\sum_{p=0}^q\nu_p\p_j^p\in R_{ij}[\p_j]$
has positive degree $q$. From Claim 1 we know that
$t_j^{-1}\p_j\cdot v$ has two choices.

In the case that $t_j^{-1}\p_j\cdot v=\lambda_j^{-1}(\p_j+a_j)v $
for some $\lambda_j\in \C^*$ and
$a_j=F_j(\p_i)=\sum_{r=0}^s\mu_r\p_i^r\in R_{ij}[\p_i]$, we have
$$\begin{aligned}0=&\lambda_i^{-k_i}\lambda_j[t_i^{k_i}\p_i,t_j^{-1}\p_j]\cdot v\\
=&\big((\p_j+F_j(\p_i-k_i))(\p_i-k_iF_i(\p_j))-(\p_i-k_iF_i(\p_j+1))(\p_j+F_j(\p_i))\big)v\\
=&k_i\p_j(F_i(\p_j+1)-F_i(\p_j))v-k_i(F_j(\p_i-k_i)-F_j(\p_i))F_i(\p_j)v\\
&-\p_i(F_j(\p_i)-F_j(\p_i-k_i))v-k_i(F_i(\p_j)-F_i(\p_j+1))F_j(\p_i)v\\
%&+k_ik_j(F_j(\p_i-k_i)F_i(\p_j)-F_j(\p_i)F_i(\p_j-k_j))v\\
=&\big(-k_i\nu_q(F_j(\p_i-k_i)-F_i(\p_i)-q)\p_j^q+\psi(\p_j)\big)v\\
=&(-k_i\nu_q(-s\mu_sk_i\p_i^{s-1}-q+\phi(\p_i))\p_j^q+\psi(\p_j))v,
\end{aligned}
$$where $\psi(\p_j)\in R_{j}[\p_j]$ has  degree $\le q-1$, and $\phi(\p_i)\in R_{ij}[\p_i]$
has degree $< s-1$. %If $F_j(\p_i)\in R_{ij}$, then the right hand side of (3.2) is $(-k_ik_jq\nu_q\p_j^q+\psi(\p_j))v\ne 0$. If $F_j(\p_i)$ has positive degree, then
%The right hand side of (3.2) is   $(k_ik_j\nu_q(-s\mu_sk_i\p_i^{s-1}-q+\phi(\p_i))\p_j^q+\psi(\p_j))v$, where.
Letting $k_i$ change in the above equality, we can see that the
right hand side  is not equal to 0, which is impossible.

 In the case that $t_j^{-1}\p_j\cdot v=\lambda_j^{-1}v$ for some $\lambda_j\in \C^*$, we have
$$\begin{aligned}0&=\lambda_i^{-1}\lambda_j[t_i\p_i,t_j^{-1}\p_j]\cdot v=((\p_i-F_i(\p_j))-(\p_i-F_i(\p_j+1))v\\
&=(F_i(\p_j+1)-F_i(\p_j))v\ne 0,\end{aligned}$$ which is absurd.

 Therefore, we must have $a_i\in \C$.

Now assume that $b_i\notin\C$ and $b_i=G_i(\p_j)\in R_{ij}[\p_j]$
has positive degree for some $j\ne i$. From Claim 1 we know that
$t_j^{-1}\p_j\cdot v$ has two choices. In the case that
$t_j^{-1}\p_j\cdot v=\lambda_j(\p_j+a_j)v$, we have
$$\begin{aligned}0=&\lambda_i^{-1}\lambda_j[t_i\p_i,t_j^{-1}\p_j]\cdot v\\
=&(\p_j+a_j)((\p_i^2-\p_i+G_i(\p_j))-(\p_i^2-\p_i+G_i(\p_j+1)))v\\
=&(\p_j+a_j)(G_i(\p_j)-G_i(\p_j+1))v\ne 0,\end{aligned}
$$which is impossible. And in the case that
$t_j^{-1}\p_j\cdot v=\lambda_j^{-1}v,$ we have
$$\begin{aligned}0&=\lambda_i^{-1}\lambda_j[t_i\p_i,t_j^{-1}\p_j]\cdot v=((\p_i^2-\p_i+F_i(\p_j))-(\p_i^2-\p_i+F_i(\p_j+1)))v\\
&=(F_i(\p_j)-F_i(\p_j+1))v\ne 0,\end{aligned}$$ which is impossible.
Therefore, we must have $b_i\in \C.$

By the same arguments above we can show that $b'_i\in \C$.

Thus the claim holds and the Lemma follows.
\end{proof}

Since $W_i^+\cong \W_1^+$ and $\C[\p_i]v$ as a $W_i^+$-module is a free $\C[\p_i]$-module of rank $1$,
by Theorem 2 and Lemma 4 we have

 \begin{lemma} For a fixed $j$ with $1\le j\le n$, there is $\lambda_j\in \C^*, a_j\in \C$ so that
\begin{equation}t_j^{k_j}\p_j\cdot v=\lambda_j^{k_j}(\p_j-k_j(a_j+1))v,\,\forall\,\, k_j\in \Z_{\ge -1}; \,\,{\rm or}\end{equation}
\begin{equation}\begin{aligned}&t_j^{-1}\p_jv=\lambda_j^{-1}v, \p_j\cdot v=\p_jv,\\
&t_j^{k_j}\p_j\cdot v=\lambda_j^{k_j}(\p_j-k_j(a_j+1))(\Pi_{p=0}^{k_j-1}(\p_j+a_j-p))v,\,\forall\,\, k_j\in \N+1.\end{aligned}\end{equation}
\end{lemma}

Using the isomorphism $\Gamma(\lambda,0)\cong \Gamma(\lambda,-1)$, we see that, in (3.3) if $a_j=0$ we can change it into $a_j=-1$ and vise versa. This will be used next. Now  we further have the following
\begin{lemma}  Let  $1\le i\ne j\le n$. If  (3.2) holds, then
\begin{equation}t_j^{k_j}\p_i\cdot v=\lambda_j^{k_j}\p_iv,\,\forall\,\, k_j\in \Z_{+}.\end{equation}
\end{lemma}

\begin{proof}
Denote $t_j\p_i\cdot v=\lambda_jH(\p_i,\p_j)v$, where $H(\p_i,\p_j)\in R_{ij}[\p_i,\p_j]$.
Using Lemma 5 and from $0=[t_l^{-1}\p_l, t_j\p_i]\cdot v$ for $1\le i\ne j\ne l\le n$
(two choices for $(t_l^{-1}\p_l)v$)  we see
that $H(\p_i,\p_j)\in \C[\p_i,\p_j]$. %We denote
% $$H(\p_i,\p_j)=\sum_{p=0}^q\phi_p(\p_j)\p_i^p\in R_{ij}[\p_i,\p_j],$$ where $\phi_p(\p_j)=\sum_{r
% =0}^{s_p}\alpha_r\p_j^r\in \C[\p_j]$ and $\phi_q(\p_j)\ne 0$.
From $\p_i\cdot v=[t_j^{-1}\p_j,$
 $t_j\p_i]\cdot v$ we deduce  that
$$\p_iv=(H(\p_i,\p_j+1)(\p_j+a_j+1)-(\p_j+a_j)H(\p_i,\p_j))v,$$
yielding that $H(\p_i,\p_j)=\p_i$. Using
 $t_j^{k_j+1}\p_i\cdot v=\frac{1}{k_j}[t_j\p_j,t_j^{k_j}\p_i]\cdot v$
  and by induction on $k_j\in \N$ we can deduce that $t_j^{k_j}\p_i\cdot v=\lambda_j^{k_j}
  \p_iv.$ So  (3.4) holds. \end{proof}

\begin{lemma}  Let  $1\le i\ne j\le n$. If   (3.3) holds, then
\begin{equation}t_j^{k_j}\p_i\cdot v=\lambda_j^{k_j}\p_i(\Pi_{p=0}^{k_j-1}(\p_j+a_j-p))v,\,\forall\,\,  k_j\in \Z_{+}.\end{equation}
\end{lemma}

\begin{proof}  Denote $t_j\p_i\cdot v=\lambda_jH(\p_i,\p_j)v$, where $H(\p_i,\p_j)\in R_{ij}[\p_i,\p_j]$.
As discussed in the proof of Lemma 6, we see
that $H(\p_i,\p_j)\in \C[\p_i,\p_j]$.
From $\p_i\cdot v=[t_j^{-1}\p_j,t_j\p_i]\cdot v$ we see that
$\p_i v=(H(\p_i,\p_j+1)-H(\p_i,\p_j))v,$ yielding that
$H(\p_i,\p_j)=(\p_j+a_j)\p_i+\psi(\p_i)$ for $\psi(\p_i)\in \C[\p_j].$

 \

{\bf{Claim 1.}}  We have $t_j\p_i\cdot v=\lambda_j(\p_j+a_j+ b )\p_iv$ for some $ b \in \C.$

We will prove this in two steps corresponding to the two choices for
$t_i^{k_i}\p_i\cdot v$ as in Lemma 5.

{\bf Case 1.} (3.2) holds with $j$ replaced by $i$.

We have
\begin{equation}\begin{aligned}&t_{j}t_i^{k_i}\p_i\cdot v
=\frac{1}{k_i}(t_j\p_i\cdot t_i^{k_i}\p_i-t_i^{k_i}\p_i\cdot t_j\p_i)\cdot v\\
=&\lambda_i^{k_i}\lambda_j(\p_i-k_i(a_i+1))\big((\p_j+a_j)+\frac{1}{k_i}(\psi(\p_i)-\psi(\p_i-k_i))\big)v.\\
\end{aligned}\end{equation}
Then we have
$$\begin{aligned}&t_{j}t_i^{k_i}\p_i\cdot v
=\frac{1}{k_i+2}(t_i^{-1}\p_i\cdot t_{j}t_i^{k_i+1}\p_i-t_{j}t_i^{k_i+1}\p_i \cdot t_i^{-1}\p_i)\cdot v\\
=&\frac{\lambda_i^{k_i}\lambda_j}{k_i+2}\Big((\p_i+1-(k_i+1)(a_i+1))(\p_i+(a_i+1))\\
&\cdot\big((\p_j+a_j)+\frac{1}{k_i+1}
(\psi(\p_i+1)-\psi(\p_i-k_i))\big)\\
&-(\p_i-(k_i+1)(a_i+1))(\p_i+a_i-k_i)\\
&\cdot\big((\p_j+a_j)+\frac{1}{k_i+1}
(\psi(\p_i)-\psi(\p_i-k_i-1))\big)\Big)v.
\end{aligned}$$
Comparing with (3.6) we see that
\begin{equation}\begin{aligned}&\frac{(k_i+1)(k_i+2)}{k_i}(\p_i-k_i(a_i+1))(\psi(\p_i)-\psi(\p_i-k_i))\\
=&(\p_i+1-(k_i+1)(a_i+1))(\p_i+(a_i+1))
(\psi(\p_i+1)-\psi(\p_i-k_i))\\
&-(\p_i-(k_i+1)(a_i+1))(\p_i+a_i-k_i)
(\psi(\p_i)-\psi(\p_i-k_i-1)),\end{aligned}\end{equation}
for all $k_i\in\Z_{\ge -1}, k_i\ne 0$. We see that (3.7) holds for all $k_i\in\Z$.
Taking $k_i=-2$ we deduce that
$\psi(\p_i+1)-\psi(\p_i+2)=\psi(\p_i)-\psi(\p_i+1)$, yielding
$\psi(\p_i)= b \p_i+ c$ for some $b,c\in \C$.
 From (3.6) with $k_i=1$ we obtain that $$t_jt_i\p_i\cdot v=\lambda_i\lambda_j(\p_i-a_i-1)(\p_j+a_j+ b )v.$$
Noting that $t_j\p_i\cdot v=\frac{1}{2}[t_i^{-1}\p_i,t_jt_i\p_i]\cdot v$, after simple computation we have
$$\lambda_j((\p_j+a_j)\p_i+( b \p_i+ c ))v=\lambda_j\p_i(\p_j+a_j+ b )v,$$
yielding $ c =0$ and $\psi(\p_i)= b \p_i$.

\

{\bf{Case 2.}} (3.3) holds with $j$ replaced by $i$.

In this case, we have\begin{equation}\begin{aligned}&t_{j}t_i^{k_i}\p_i\cdot v%=&\frac{1}{k_i}[t_j\p_i, t_i^{k_i}\p_i]\cdot v.\\
=\frac{1}{k_i}(t_j\p_i\cdot t_i^{k_i}\p_i-t_i^{k_i}\p_i\cdot t_j\p_i)\cdot v\\
=&\lambda_i^{k_i}\lambda_j(\p_i-k_i(a_i+1))\big((\p_j+a_j)+\frac{1}{k_i}(\psi(\p_i)-\psi(\p_i-k_i))\big)\\
&\cdot \Pi_{p=0}^{k_i-1}(\p_i+a_i-p) v.\\
\end{aligned}\end{equation}
Then
$$\begin{aligned}&t_{j}t_i^{k_i}\p_i\cdot v
=\frac{1}{k_i+2}(t_i^{-1}\p_i\cdot t_{j}t_i^{k_i+1}\p_i-t_{j}t_i^{k_i+1}\p_i \cdot t_i^{-1}\p_i)\cdot v\\
=&\frac{\lambda_i^{k_i}\lambda_j}{k_i+2}\Big((\p_i+1-(k_i+1)(a_i+1))(\p_i+(a_i+1))\\
&\cdot\big((\p_j+a_j)+\frac{1}{k_i+1}
(\psi(\p_i+1)-\psi(\p_i-k_i))\big)\\
&-(\p_i-(k_i+1)(a_i+1))(\p_i+a_i-k_i)\\
&\cdot\big((\p_j+a_j)+\frac{1}{k_i+1}
(\psi(\p_i)-\psi(\p_i-k_i-1))\big)\Big)\Pi_{p=0}^{k_i-1}(\p_i+a_i-p)v.
\end{aligned}$$ Comparing with equation (3.8) we also have that equation (3.7) holds. By similar arguments we have $\psi(\p_i)= b \p_i$.  Claim 1 follows.

\

 If $ b =0$, from $t_j^{k_j}\p_i\cdot v=[t_j^{k_j-1}\p_j,t_j\p_i]\cdot v$ with $ k_j\in \N,$ we deduce (3.5).
Next we assume that $ b \ne 0$.
For $k_j\ge2$ we have
\begin{equation}\begin{aligned}&t_j^{k_j}\p_i\cdot v%=[t_j^{k_j-1}\p_j,t_j\p_i]\cdot v
=(t_j^{k_j-1}\p_j\cdot t_j\p_i-t_j\p_i\cdot t_j^{k_j-1}\p_j)\cdot v\\
=&\lambda_j^{k_j}\p_i\big(\Pi_{p=1}^{k_j-2}(\p_j+a_j-p)\big)\Big(\big((\p_j+a_j-k_j+1)\\
&\cdot(\p_j+a_j)(\p_j-(k_j-1)(a_j+1))\\
&-(\p_j+a_j-k_j+1)(\p_j-(k_j-1)(a_j+1)-1)(\p_j+a_j)\big)\\
&+ b \big((\p_j+a_j)(\p_j-(k_j-1)(a_j+1))\\
&-(\p_j+a_j-k_j+1)(\p_j-(k_j-1)(a_j+1)-1)\big)\Big)v\\
=&\lambda_j^{k_j}\p_i\big( b  k_j(\p_j-(k_j-2)(a_j+1)-1)\Pi_{p=1}^{k_j-2}(\p_j+a_j-p)\\
&+\Pi_{p=0}^{k_j-1}(\p_j+a_j-p)\big)v.
\end{aligned}\end{equation}
Using this formula we see that
$$\begin{aligned}&t_j^{k_j+1}\p_i\cdot v=\frac{1}{k_j}(t_j\p_j\cdot t_j^{k_j}\p_i-t_j^{k_j}\p_i\cdot t_j\p_j )\cdot v\\
=&\lambda_j^{k_j+1}\p_i\Big(\Pi_{p=0}^{k_j}(\p_j+a_j-p)\\
&+ b  (k_j+1)(\p_j-(k_j-1)(a_j+1)-1)\cdot\Pi_{p=1}^{k_j-1}(\p_j+a_j-p)\\
&+2 b a_j(a_j+1)(k_j-1)(2\p_j+2a_j-k_j)\Pi_{p=2}^{k_j-2}(\p_j+a_j-p)\Big)v.
\end{aligned}$$ Comparing with equation (3.9) we deduce that
$$2 b \big(a_j(a_j+1)(k_j-1)(2\p_j+2a_j-k_j)\big)\Pi_{p=2}^{k_j-2}(\p_j+a_j-p)=0,{\rm for }\,\,k_j\ge2,$$
which forces that $$a_j(a_j+1)(k_j-1)=0.$$
Hence $a_j=0$ or $-1.$ Formula (3.9) becomes
\begin{equation}t_j^{k_j}\p_i\cdot v=\lambda_j^{k_j}\p_i(\p_j+k_j(a_j+ b ))\Pi_{p=1}^{k_j-1}(\p_j-p)v, \ {\rm for }\,\,k_j\ge2.\end{equation}

\

{\bf{Claim 2.}} If $a_j=0,$ then $ b =-1$; and if $a_j=-1,$ then $ b =1.$

%\begin{equation}\begin{aligned}&t_it_j\p_i\cdot v=(t_j\p_i\cdot t_i\p_i-t_i\p_i\cdot t_j\p_i)\cdot %v,\\
%&t_it_j\p_j\cdot v=(t_i\p_j\cdot t_j\p_j-t_j\_j\cdot t_i\p_j)\cdot v.\end{aligned}\end{equation}
We will prove this claim in two cases corresponding to the two choices for $t_i^{k_i}\p_i\cdot v$ as in Lemma 5.

{\bf{Case 1.}} (3.2) holds with $j$ replaced by $i$.

In this case, by Lemma 6, equations (3.10) and  (3.11) we have
$$t_it_j\p_i\cdot v=(t_j\p_i\cdot t_i\p_i-t_i\p_i\cdot t_j\p_i)\cdot v
=\lambda_i\lambda_j(\p_i-(a_i+1))(\p_j+(a_j+ b ))v,
$$
$$t_it_j\p_j\cdot v=(t_i\p_j\cdot t_j\p_j-t_j\_j\cdot t_i\p_j)\cdot v
=\lambda_i\lambda_j\p_j(\p_j-1)v.
$$ Then
$$\begin{aligned}\lambda_i\lambda_j&\big((\p_i-(a_i+1))(\p_j+(a_j+ b ))-\p_j(\p_j-1)\big)v\\
=&t_it_j(\p_i-\p_j)\cdot v=[t_i\p_j,t_j\p_i]\cdot v\\
=&\lambda_i\lambda_j(\p_i-\p_j)(\p_j+a_j+ b )v,
\end{aligned}$$
which means that $$(a_j+ b )-a_i=(a_i+1)(a_j+ b )=0.$$
We easily deduce that if $a_j=0,$ then $ b =-1$; and if $a_j=-1,$ then $ b =1.$ The claim follows in this case.

{\bf{Case 2.}} Equation (3.3) holds with $j$ replaced by $i$.

By Claim 1 we know that
$$t_i\p_j\cdot v=\lambda_i\p_j(\p_i+a_i+c_i)v$$
for some $c_i\in \C.$ Then
$$t_it_j\p_i\cdot v=\lambda_i\lambda_j(\p_i+a_i)(\p_i-(a_i+1))(\p_j+a_j+ b )v,$$
$$t_it_j\p_j\cdot v=\lambda_i\lambda_j(\p_j+a_j)(\p_j-(a_j+1))(\p_i+a_i+c_i)v.$$
So
$$\begin{aligned}\lambda_i\lambda_j&\big((\p_i+a_i)(\p_i-(a_i+1))(\p_j+a_j+ b )\\
&-(\p_j+a_j)(\p_j-(a_j+1))(\p_i+a_i+c_i)\big)v\\
=&t_it_j(\p_i-\p_j)\cdot v=[t_i\p_j,t_j\p_i]\cdot v\\
=&\lambda_i\lambda_j(\p_i-\p_j)(\p_i+a_i+c_i)(\p_j+a_j+ b )v.
\end{aligned}$$ Computing the coefficients of $\p_i\p_j$ and $\p_i$ we see that
$$(a_j+ b )-(a_i+c_i)=0,\,\,\, a_ia_j+a_ib-a_j^2+a_jc+bc+b=0.$$
Noting that $a_j=0,-1$, we  deduce that if $a_j=0,$ then $ b =-1$; and if $a_j=-1,$ then $ b =1.$ The claim follows.

\

First we assume that  $a_j=0$ and $ b =-1$. we can rewrite (3.10) as follows   $$t_j^{k_j}\p_i\cdot v=\lambda_j^{k_j}\p_i\Pi_{p=1}^{k_j}(\p_j-p)v,\ k_j\in \N.$$
As in the remark after Lemma 3.5, we retake $a_j=-1$. Then $b$ is changed to $0$ in Claim 1 and (3.5) follows with the new value of $a_j$ in this case.

At last, we assume that  $a_j=-1$ and $ b =1$. we can rewrite (3.10) as follows   $$t_j^{k_j}\p_i\cdot v=\lambda_j^{k_j}\p_i\Pi_{p=0}^{k_j-1}(\p_j-p)v,\ k_j\in \N.$$
As in the remark after Lemma 3.5, we retake $a_j=0$. Then $b$ is changed to $0$ in Claim 1 and (3.5) follows with the new value of $a_j$ in this case. This completes the proof.
\end{proof}

Now we are ready to prove our main result in this paper.

\begin{theorem}Let $M$ be a $\W_n^+$-module which is a free $U(\h_n)$-module of  rank $1$. Then
$M\cong \Omega(\Lambda_{n},a, S) $ for some
$\Lambda_n=(\lambda_1,\cdots,\lambda_n) \in (\C^*)^n, a\in \C$, and
$S\subset \{1,2,..., n\}$.
\end{theorem}

\begin{proof}
Since $M$ is a free $U(\h_n)$-module of rank 1, there is a nonzero element $v\in M$
such that $M=U(\h_n)v=\C[\partial_1,\partial_2,\cdots,\partial_n]v.$
Since $\W_n^{+}$ is generated by $t_i^{k_i}\p_i, t_j^{k_j}\p_i$ for $ 1\le
i\ne j\le n, k_i\in \Z_{\ge -1}, k_j\in \Z_+,$ by Lemma 1 we know
 that the structure of the $\W_n^{+}$-module $M$ is determined by the action
  of the elements $t_i^{k_i}\p_i, t_j^{k_j}\p_i$ for $ 1\le i\ne j\le n, k_i\in \Z_{\ge -1}, k_j\in \Z_+$ on $v$.

From Lemmata 5, 6 and 7, there is $\lambda_j\in \C^*, a_j\in \C$ so that Lemmata 5, 6, and 7 hold.
Let $S$ be the subset of $\{1,2,...,n\}$ consisting of all $j$ such that (3.3) holds. Next we need only to prove that
$a_i=a_j$ for all   $1\le i\ne j\le n$. We do this in three cases:  (i). $i, j\notin S$;
(ii).  $i, j\in S$, (iii). $i\notin S$ and $j\in S$.

For Case (i), by Lemmata 5 and 6 we deduce that
$$\begin{aligned}
t_it_j\p_i\cdot v
&=(t_j\p_i\cdot t_i\p_i-t_i\p_i\cdot t_j\p_i)\cdot v
=\lambda_i\lambda_j(\p_i-(a_i+1))v,\\
t_it_j\p_j\cdot v
&=\lambda_i\lambda_j(\p_j-(a_j+1))v.\end{aligned}$$ We have
$$\lambda_i\lambda_j(\partial_j-\partial_i+(a_i-a_j))v
=t_it_j(\p_j-\p_i)\cdot v
=[t_j\p_i, t_i\p_j]\cdot v
 %=&(t_j^{k_j}\p_i\cdot t_i^{k_i}\p_i-t_i^{k_i}\p_i\cdot t_j^{k_j}\p_i)\cdot v\\
=\lambda_i\lambda_j(\partial_j-\partial_i)v,
$$ yielding that $a_i=a_j.$

For Case (ii), by Lemmata 5 and 7 we deduce that
$$\begin{aligned}
t_it_j\p_i\cdot v
&=(t_j\p_i\cdot t_i\p_i-t_i\p_i\cdot t_j\p_i)\cdot v\\
&=\lambda_i\lambda_j(\p_i-(a_i+1))(\p_i+a_i)(\p_j+a_j)v,\\
t_it_j\p_j\cdot v
&=\lambda_i\lambda_j(\p_j-(a_j+1))(\p_i+a_i)(\p_j+a_j)v.\end{aligned}$$ We have
$$\begin{aligned}\lambda_i\lambda_j&(\partial_j-\partial_i+(a_i-a_j))(\p_i+a_i)(\p_j+a_j)v
=t_it_j(\p_j-\p_i)\cdot v\\
&=[t_j\p_i, t_i\p_j]\cdot v\\
 %=&(t_j^{k_j}\p_i\cdot t_i^{k_i}\p_i-t_i^{k_i}\p_i\cdot t_j^{k_j}\p_i)\cdot v\\
&=\lambda_i\lambda_j(\partial_j-\partial_i)(\p_i+a_i)(\p_j+a_j)v,\end{aligned}
$$ yielding that $a_i=a_j.$

For Case (iii), by Lemmata 5, 6 and 7 we deduce that
$$\begin{aligned}
t_it_j\p_i\cdot v
&=(t_j\p_i\cdot t_i\p_i-t_i\p_i\cdot t_j\p_i)\cdot v\\
&=\lambda_i\lambda_j(\p_i-(a_i+1))(\p_j+a_j)v,\\
t_it_j\p_j\cdot v
&=\lambda_i\lambda_j(\p_j-(a_j+1))(\p_j+a_j)v.\end{aligned}$$ We have
$$\begin{aligned}\lambda_i\lambda_j&(\partial_j-\partial_i+(a_i-a_j))(\p_j+a_j)v
=t_it_j(\p_j-\p_i)\cdot v\\
&=[t_j\p_i, t_i\p_j]\cdot v\\
 %=&(t_j^{k_j}\p_i\cdot t_i^{k_i}\p_i-t_i^{k_i}\p_i\cdot t_j^{k_j}\p_i)\cdot v\\
&=\lambda_i\lambda_j(\partial_j-\partial_i)(\p_j+a_j)v,\end{aligned}
$$ yielding that $a_i=a_j.$

So $a_i=a_j$ for all $ 1\le i\ne j\le n.$ Denote $a=a_i, 1\le i\le n.$
 We know that the action
  of the elements $t_i^{k_i}\p_i, t_j^{k_j}\p_i$ for $ 1\le i\ne j\le n, k_i\in \Z_{\ge -1}, k_j\in \Z_+$ on $v$ is the same
  as the action on $1$ in $\Omega(\Lambda_{n},a, S) $. Thus $M\cong \Omega(\Lambda_{n},a, S) $.\end{proof}

For $\W_n$-module structure on $U(\h_n)$, we have the following

\begin{theorem}Let $M$ be a $\W_n$-module which is a free $U(\h_n)$-module of  rank $1$. Then
$M\cong \Omega(\Lambda_n,a)$ for some $\Lambda_n\in (\C^*)^n$ and some $a\in
\C.$
\end{theorem}

\begin{proof}
Since $M$ is a free $U(\h_n)$-module of rank 1, there is a nonzero element $v\in M$
such that $M=U(\h_n)v=\C[\partial_1,\partial_2,\cdots,\partial_n]v.$
We know that $\W_n$ has two subalgebras isomorphic to $\W_n^{+}$:
$$L_1=\span\{t^k\partial_i:k\in \cup_{i=1}^n\Z_{+,i}^n, 1\le i\le n\}, $$
$$L_2=\span\{t^k\partial_i:-k\in \cup_{i=1}^n\Z_{+,i}^n, 1\le i\le n\}$$
which share the same Cartan subalgebra $\h_n$.

If we consider $M$ as an $L_1$-module, from Theorem 8 there exist  some
$\Lambda_n=(\lambda_1,\cdots,\lambda_n) \in (\C^*)^n, a\in \C$, and
$S\subset \{1,2,..., n\}$ so that (1.3) holds.
If we consider $M$ as an $L_2$-module, there exist  some
$\Lambda'_n=(\lambda'_1,\cdots,\lambda'_n) \in (\C^*)^n, a'\in \C$, and
$S'\subset \{1,2,..., n\}$ so that (1.3) holds with the corresponding parameters and with $-k\in\Z_{+,i}^n$. Then $S=S'=\varnothing$,  $a=a'$ and $\Lambda_n=\Lambda'_n$.
Therefore (1.2) holds, which means $M\cong \Omega(\Lambda_n,a)$ as $\W_n$-module.  This completes the proof.
\end{proof}

One can easily check that all the proofs in this section are valid if $n$ is replaced by $\infty$. So we have the following

\begin{theorem} \begin{itemize}\item[(a).] Let $V$ be a $\W_{\infty}^{+}$-module which is a free $U(\h_{\infty})$-module of  rank $1$. Then
$V\cong \Omega(\Lambda_{\infty},a,S)$ for some $ \Lambda_{\infty}\in (\C^*)^{\infty}$, some $a\in\C$ and some subset $S$ of $\N$.
 \item[(b).] Let $M$ be a $\W_{\infty}$-module which is a free $U(\h_{\infty})$-module of rank $1$. Then
$M\cong \Omega(\Lambda_{\infty},a)$ for some sequence $\Lambda_{\infty}$ of $\C^*$ and some $a\in \C$.
 \end{itemize}\end{theorem}

For simplicity of these modules we have the following

\begin{theorem}  Let $ \Lambda_{n}\in (\C^*)^{n}$,  $a\in\C$ and  $S\subset\{1,2,...,n\}$.
  \begin{itemize}\item[(a).]
 $\W_{n}^{+}$-module $\Omega(\Lambda_{n},a,S)$ is  simple if and only if $S\neq\varnothing$, or  $S=\varnothing$ and $a\ne -1$.
 \item[(b).] $\W_{n}$-module $\Omega(\Lambda_{n},a)$ is simple if and only if and $a\ne -1$.
 \end{itemize}\end{theorem}

\begin{proof} (a). If $S\neq\varnothing$, from Theorem 33(ii) in \cite{N} we know that $\Omega(\Lambda_{n},a,S)$ is  simple as an $\sl_{n+1}$-module. So it is simple as a $\W_n^+$-module since $\sl_{n+1}$ is a subalgebra of $\W_n^+$.

If $S=\varnothing$ and $a= -1$, from Example 3 in \cite{TZ3} we know that $\Omega(\Lambda_{n},a,S)$ is  not simple as a $\W_n$-module. So it is not simple as a $\W_n^+$-module
since $\W_n^+$ is a subalgebra of $\W_n$.

If $S=\varnothing$ and $a\ne -1$, from Lemma 1 and by equation (3.2), we can deduce that each nonzero element of $\Omega(\Lambda_{n},a,\varnothing)$ can generate $1$. Hence $\Omega(\Lambda_{n},a,\varnothing)$ as a $\W_n^+$-module is simple.

%{\bf We need to prove for this case.}

Part (b) follows from Example 3 in \cite{TZ3}.
\end{proof}
We remark that, under the well-known embedding of (1.1) in \cite{TZ3}, the $\sl_{n+1}$-modules restricted from the simple $\W_n^+$-modules constructed in Theorem 8 exhaust  all the $\sl_{n+1}$-modules whose are a free module over the Cartan subalgebra of $\sl_n$ classified in \cite{N}.

\

\noindent {\bf Acknowledgement.} The second author is partially
supported by NSF of China (Grant 11271109), NSERC, and University Research Professor grant at Wilfrid Laurier University.
The authors thank Prof. R. Lu for a lot of helpful discussions during the
preparation of the paper.

\vspace{10mm}

\noindent H. Tan: Department of Applied Mathematics, Changchun University of Science and Technology, Changchun, Jilin,
130022, P.R. China.
and College of Mathematics and Information Science,
Hebei Normal (Teachers) University, Shijiazhuang, Hebei, 050016 P.
R. China. Email: hjtan9999@yahoo.com

\vspace{0.2cm}
 \noindent K. Zhao: Department of Mathematics, Wilfrid
Laurier University, Waterloo, ON, Canada N2L 3C5, and College of
Mathematics and Information Science, Hebei Normal (Teachers)
University, Shijiazhuang, Hebei, 050016 P. R. China. Email:
kzhao@wlu.ca

\end{document}